\documentclass[10pt]{amsart}

\usepackage[centertags]{amsmath}
\usepackage{amsfonts}
\usepackage{amssymb}
\usepackage{amsthm}
\usepackage[bookmarks=true,hyperindex,pdftex,colorlinks,citecolor=red, linkcolor=blue]{hyperref}
\usepackage{tikz}
\usepackage{tikz-cd}
\usepackage[normalem]{ulem}
\usepackage[shortlabels]{enumitem} 
\usepackage{marginnote}
\usepackage[all]{xy}
\usepackage[dvipsnames]{xcolor}
\usepackage{bm}
\usepackage{cancel}

\newcommand{\Ree}{\operatorname{Re}}

\newcommand{\vertiii}[1]{{\left\vert\kern-0.25ex\left\vert\kern-0.25ex\left\vert #1 
		\right\vert\kern-0.25ex\right\vert\kern-0.25ex\right\vert}}

\newcommand{\R}{\mathbb{K}}

\newcommand{\restricted}{\mathord{\upharpoonright}}

\newcommand{\pten}{\widehat{\otimes}_{\pi}}
\newcommand{\borel}{\mathrm{Borel}}
\newcommand{\baire}{\mathrm{Baire}}

\newcommand\restr[2]{{
		\left.\kern-\nulldelimiterspace 
		#1 
		\littletaller 
		\right|_{#2} 
}}
\newcommand{\littletaller}{\mathchoice{\vphantom{\big|}}{}{}{}}

\definecolor{egraf}{rgb}{0.2,0.4,0}

\def\<{\langle}
\def\>{\rangle}

\newcommand{\NA}{\operatorname{NA}}
\newcommand{\INA}{\operatorname{INA}}

\DeclareMathOperator{\co}{co}
\DeclareMathOperator{\aco}{aco}

\DeclareMathOperator{\Span}{span}

\theoremstyle{plain}
\newtheorem{thm}{Theorem}[section]
\newtheorem{introthm}{Theorem}

\newtheorem{introcor}[introthm]{Corollary}
\newtheorem{theorem}[thm]{Theorem}

\newtheorem{lemma}[thm]{Lemma}

\newtheorem{proposition}[thm]{Proposition}

\theoremstyle{definition}

\newtheorem{remark}[thm]{Remark}

\author[M. Jung]{Mingu Jung}
\address[M. Jung]{Department of Mathematics \& Research Institute for Natural Sciences, Hanyang University, 04763 Seoul, Republic of Korea}
\email{mingujung@hanyang.ac.kr}

\begin{document}
	\title[Optimal integral representations in projective tensor products]{Optimal integral representations in projective tensor products: countability and topology}

	\subjclass[2020]{46B28, 46B03, 46B20, 46G10}

	\keywords{projective tensor product; optimal tensor representation; Bochner integral; weak and weak-star Borel structures}
	
\begin{abstract}
We study optimal integral representations in projective tensor products, focusing on two questions: whether they can always be replaced by countable optimal decompositions, and whether the resulting notion depends on the Borel topology used on the product of the unit balls. We show that every infinite-dimensional Banach space $X$ admits an equivalent norm for which, denoting the resulting space by $Z$, there exists an affine homeomorphic embedding
\[
\Psi : \mathcal{P}([0,1]) \to S_{Z \pten Z}
\]
such that $\Psi (\mathcal{P}([0,1]) ) \subseteq \INA_\pi (Z\pten Z)$ and 
\[
\Psi (\alpha) \in \NA_\pi(Z \pten Z) \iff \text{$\alpha$ is countably supported}.
\]
In particular, this implies that there exists a tensor 
\[
u \in \operatorname{INA}_{\pi}(Z\widehat\otimes_{\pi} Z)
 \setminus \operatorname{NA}_{\pi}(Z\widehat\otimes_{\pi} Z)
\]
and the witnessing measure may be chosen to be a nonatomic Radon probability measure. The construction realizes an affine copy of $\mathcal P([0,1])$ as an exposed face of $B_Z$ and compares the diagonal Lebesgue coupling with countable mixtures of product measures. We also prove that the norm, weak, and--on dual spaces--weak-star versions of integral projective norm attainment define the same class of tensors. Consequently, every infinite-dimensional separable reflexive Banach space $X$ with the approximation property admits an equivalent norm such that, for the resulting space $Z$,
$\operatorname{NA}_{\pi}(Z\widehat\otimes_{\pi} Z)
 \subsetneq
 \operatorname{INA}_{\pi}(Z\widehat\otimes_{\pi} Z)
 = Z\widehat\otimes_{\pi} Z$.
\end{abstract}

\maketitle

\section{Introduction}

Let $X$ and $Y$ be Banach spaces over either the real field or the complex field. An element $u \in X \pten Y$ is said to be \emph{projective norm-attaining} if it admits an optimal representation
\begin{equation}\label{def-NA-pi}\tag{$\dagger$}
u=\sum_{n=1}^{\infty}\lambda_n x_n\otimes y_n,
\quad
(x_n,y_n)\in B_X\times B_Y,\quad
\lambda_n\ge 0, \quad \sum_{n=1}^{\infty}\lambda_n=\|u\|_\pi,
\end{equation}
where $\|u\|_\pi$ denotes the projective tensor norm of $u$. The set of all projective norm-attaining tensors is denoted by $\NA_\pi(X\pten Y)$.

The notion was introduced in \cite{DJRR} in connection with norm attainment for nuclear operators and with Bishop--Phelps-type questions for operators. Its symmetric counterpart was subsequently studied in \cite{DGJR}, while related finite-dimensional work dates back to the work of Pe{\l}czy\'nski and Tomczak-Jaegermann \cite{PT}. Recent work has identified broad classes of projective tensor products in which every element is norm-attaining \cite{GGR} and has provided concrete descriptions of optimal representations in spaces of Bochner integrable functions \cite{GG}. 

The existence of an optimal representation is a natural tensorial analogue
of norm attainment for operators and functionals in the classical
Bishop--Phelps theory. Moreover, projective norm-attaining tensors are
closely related to optimal decompositions in projective tensor products,
the geometry of tensor norms, and the theory of nuclear operators.

Inspired by recent developments in the extremal structure and Choquet theory of Lipschitz-free spaces \cite{APS,APS2}, one can replace countable tensor representations \eqref{def-NA-pi} by integral representations, leading to the notion of integral projective norm-attainment introduced in \cite{ADGJR}. An element $u\in X\pten Y$ is called an \emph{integral projective norm-attaining tensor} if there exists a finite positive Borel measure $\mu$ on $B_X\times B_Y$ such that the mapping
\[
\varphi:B_X\times B_Y\to X\pten Y,
\qquad
\varphi(x,y)=x\otimes y,
\]
is $\mu$-Bochner integrable and
\begin{equation}\label{def:eq-INA-pi}\tag{$\dagger\dagger$}
u=\int_{B_X\times B_Y} \varphi(x,y)\,d\mu(x,y),
\qquad
\|u\|_\pi=\|\mu\|.
\end{equation}
Unless otherwise stated, $B_X\times B_Y$ is endowed with the product of the norm topologies, and all Borel measures on this product are taken with respect to the corresponding Borel $\sigma$-algebra. We denote by $\INA_\pi(X\pten Y)$ the set of all integral projective norm-attaining tensors.

Every optimal series representation yields an integral representation by means of a countably supported measure, hence 
\[
\NA_\pi(X\pten Y)\subseteq \INA_\pi(X\pten Y).
\]
Moreover, integral projective norm-attaining tensors can be approximated in norm by projective norm-attaining tensors with finite representations \cite[Theorem 2.11]{ADGJR}. However, this approximation does not decide whether an individual optimal integral representation can always be replaced by an optimal countable one. This led to the following question in \cite[Question 2.2]{ADGJR}:
\[
\NA_\pi(X\pten Y) \stackrel{?}{=} \INA_\pi(X\pten Y) \, \text{ for all Banach spaces } X, Y.
\]

Our main result shows that the answer is negative. Let $\mathcal{P}([0,1])$ denote the space of Borel probability measures on $[0,1]$ equipped with the weak-star topology. We say that $\alpha\in\mathcal P([0,1])$ is \emph{countably supported} if it is concentrated on a countable subset of $[0,1]$, or equivalently, if
\[
    \alpha=\sum_{n=1}^\infty\lambda_n\delta_{t_n}
\]
for some $t_n\in[0,1]$ and $\lambda_n\geq0$ with $\sum_{n=1}^\infty\lambda_n=1$.

\begin{introthm}\label{thm:mainA}
Every infinite-dimensional Banach space $X$ admits an equivalent norm $\vertiii{\cdot}$ such that for $Z = (X, \vertiii{\cdot})$, there exists an affine homeomorphic embedding
\[
\Psi : \mathcal{P}([0,1]) \to S_{Z \pten Z}
\]
satisfying that $\Psi (\mathcal{P}([0,1]) ) \subseteq \INA_\pi (Z\pten Z)$ and 
\[
\Psi (\alpha) \in \NA_\pi(Z \pten Z) \iff \text{$\alpha$ is countably supported}.
\]
\end{introthm}

\begin{introcor}\label{cor:mainA}
    An infinite-dimensional Banach space $X$ admits an equivalent norm $\vertiii{\cdot}$ such that for $Z = (X, \vertiii{\cdot})$, there exists a tensor 
\[
u \in  \INA_\pi (Z \pten Z) \setminus \NA_\pi (Z \pten Z).
\]
Moreover, $u$ is witnessed by a nonatomic Radon probability measure.
\end{introcor}

The construction has a measure-theoretic mechanism. Using a normalized basic sequence, we embed the simplex $\mathcal P([0,1])$ affinely and homeomorphically into the ambient Banach space by recording its moments. We then renorm a one-dimensional extension so that this copy of $\mathcal P([0,1])$ becomes an exposed face of the new unit ball. For each $\alpha\in\mathcal P([0,1])$, the tensor $\Psi(\alpha)$ is the barycenter of the diagonal curve
\[
    t\mapsto p(t)\otimes p(t)
\]
with respect to $\alpha$ (see \eqref{def:psi(alpha)}). The tensor in Corollary~\ref{cor:mainA} is obtained by taking $\alpha$ to be the Lebesgue probability measure.

If $\Psi(\alpha)$ admitted an optimal countable representation, the exposed-face structure would force every summand to correspond to a pair of probability measures $\alpha_n,\beta_n\in\mathcal P([0,1])$. The resulting representation would therefore determine a countable convex combination of product measures $\alpha_n\times\beta_n$ on $[0,1]^2$. The moment coordinates force this measure to coincide with the diagonal coupling $\Delta_\sharp\alpha$. Since every product probability measure concentrated on the diagonal is a Dirac mass (Lemma \ref{lem:diagonal-product}), $\Delta_\sharp\alpha$, and hence $\alpha$, must be countably supported.

A recent preprint of Han, Kim, Mart\'in, and Rueda Zoca provides another negative answer to \cite[Question~2.2]{ADGJR}. More precisely, they proved in \cite[Corollary~4.9]{HKMRZ26} that, whenever $1<p,q<\infty$ and $1/p+1/q<1$, one has
\[
    \NA_\pi(\ell_p\widehat{\otimes}_\pi\ell_q) \subsetneq \INA_\pi(\ell_p\widehat{\otimes}_\pi\ell_q) = \ell_p\widehat{\otimes}_\pi\ell_q.
\]
Their result gives concrete examples among classical reflexive spaces,
without passing to an equivalent renorming, whereas Theorem \ref{thm:mainA} applies, after an equivalent renorming, to every infinite-dimensional Banach space and identifies an affine copy of $\mathcal P([0,1])$ on which projective norm attainment is characterized exactly by countable support.

The second part of the paper concerns topological variants of integral projective norm-attaining tensors. Following~\cite[Section~3]{ADGJR}, one may define integral projective norm attainment using the norm, weak, or weak-star Borel structure on the product of the unit balls. Since these Borel structures need not coincide globally, the corresponding classes could a priori be different. Nevertheless, we prove that Bochner integrability of $\varphi$ witnessing the well-definedness of the representation \eqref{def:eq-INA-pi} removes this dependence on topology.

\begin{introthm}\label{thm:mainB}
    Let $X$ and $Y$ be Banach spaces. Then
    \begin{enumerate}[label=(\alph*)]
        \itemsep0.25em
        \item $\INA_w(X\pten Y)=\INA_\pi(X\pten Y).$
        \item $\INA_{w^*} (X^*\pten Y^*)=\INA_w(X^*\pten Y^*)=\INA_\pi(X^* \pten Y^*).$
        \end{enumerate}
\end{introthm}

The key point is local rather than global. A Bochner-integrable tensor-valued map is essentially separably valued. This implies that, modulo a null set, the representing measure is carried by $B_{X_0}\times B_{Y_0}$ for suitable separable subspaces $X_0\subseteq X$ and $Y_0\subseteq Y$. On this localized support, the weak and norm Borel structures coincide by a result of Edgar \cite{Edgar}. In the dual setting, Talagrand's theorem in \cite{Talagrand} gives the corresponding coincidence for the relative weak-star Borel structure on norm-separable subspaces.

The two relaxations considered here therefore behave in opposite ways: allowing arbitrary positive representing measures genuinely enlarges the class of projective norm-attaining tensors, whereas weakening the underlying Borel topology does not, once Bochner integrability is imposed.

Combining Theorem~\ref{thm:mainB} with the weak-star representation theorem \cite[Proposition~3.2]{ADGJR} gives the following consequence, which shows that the separation in Theorem~\ref{thm:mainA} can occur even when every tensor is integral projective norm-attaining.

\begin{introcor}\label{cor:mainB}
Let $X$ be an infinite-dimensional separable reflexive Banach space with the approximation property. Then $X$ admits an equivalent norm such that, for the resulting space $Z$,
\[
 \operatorname{NA}_{\pi}(Z\widehat\otimes_{\pi}Z)
 \subsetneq
 \operatorname{INA}_{\pi}(Z\widehat\otimes_{\pi}Z)
 =
 Z\widehat\otimes_{\pi}Z.
\]
\end{introcor}

The paper is organized as follows. In Section~\ref{sec:renorming}, we construct the moment embedding of $\mathcal P([0,1])$ and realize its image as an exposed face of an equivalent unit ball. Section~\ref{sec:no-countable-representation} introduces the diagonal tensor and proves Theorem~\ref{thm:mainA} by comparing the diagonal coupling with countable mixtures of product measures. In Section~\ref{sec:topological-variants} we prove the topology-independence result and derive Corollary~\ref{cor:mainB}.

\section{Moment embeddings and exposed-face renormings}\label{sec:renorming}

The purpose of this section is to construct, inside an arbitrary infinite-dimensional Banach space after an equivalent renorming, an exposed face affinely homeomorphic to the space $\mathcal P([0,1])$ of Borel probability measures on $[0,1]$. The moment coordinates below serve two roles: they separate probability measures and, later, recover a measure on $[0,1]^2$ from the associated tensor. The exposed-face property will force every elementary tensor occurring in an optimal countable representation to correspond to a pair of points in this copy of $\mathcal P([0,1])$.

\subsection{The moment embedding of $\mathcal{P}([0,1])$}
Let $X$ be an infinite-dimensional Banach space. Let $(e_n)$ be a normalized basic sequence in $X$, and let $E := \overline{\Span} \{e_n : n \in \mathbb{N}\} \subseteq X$. Let $(e_n^*) \subseteq E^*$ be the coefficient functionals, and extend each $e_n^*$ to $X^*$ by the Hahn-Banach theorem.

For each $t\in [0,1]$, define
\begin{equation}\label{eq:def-phi}
  \phi(t): = \sum_{n=1}^\infty \frac{1}{2^n} t^n e_n \in X.
\end{equation}
Notice that $\phi: [0,1] \to X$ is norm continuous and $\alpha$-Bochner integrable for every $\alpha \in \mathcal{P}([0,1])$. 

Define
\begin{equation}\label{eq:def-Phi}
  \Phi:\mathcal{P}([0,1])\to X,
  \qquad
  \Phi(\alpha):= \sum_{k=1}^\infty \frac{1}{2^k} \left( \int_{[0,1]} t^k \, d\alpha(t)\right) e_k.
\end{equation}
Note that $\Phi(\delta_t) = \phi(t)$ for every $t \in [0,1]$, where $\delta_t$ is the Dirac mass at $t$. We equip $\mathcal{P}([0,1]) \hookrightarrow C([0,1])^*$ with the weak-star topology.

\begin{lemma}\label{lem:relation-between-Phi-phi}
 Let $\phi$ and $\Phi$ be defined as in \eqref{eq:def-phi} and \eqref{eq:def-Phi}, respectively. Given $\alpha \in \mathcal{P}([0,1])$, 
    \[
    \Phi(\alpha) = \int_{[0,1]} \phi(t) \, d\alpha(t).
    \]
\end{lemma}

\begin{proof}
For each $n\in\mathbb{N}$, the map $a \in \mathbb{K} \mapsto a e_n \in X$ is bounded linear; hence it commutes with Bochner integration. That is,
\[
\int_{[0,1]} t^n e_n \, d\alpha(t)= \left( \int_{[0,1]} t^n \, d\alpha(t) \right) e_n.
\] 
Since the series defining $\phi$ converges uniformly on $[0,1]$, it may be integrated term by term. Thus, 
\begin{align*}
\Phi(\alpha) =   \int_{[0,1]} \left( \sum_{k=1}^\infty \frac{1}{2^k }t^k e_k \,  \right) d\alpha(t) = \int_{[0,1]} \phi(t)\,d\alpha(t).
\end{align*}
This completes the proof.
\end{proof}

\begin{lemma}\label{lem:barycenter_continuity}
Let $K$ be a compact Hausdorff space, let $E$ be a Banach space, and let $f\colon K\to E$ be norm continuous. Then the map
\[
    \mu \in \mathcal{P}(K) \mapsto \int_K f(t)\,d\mu(t),
\]
is continuous from the weak-star topology of $\mathcal P(K)$ to the norm topology of $E$.
\end{lemma}

\begin{proof}
Fix $\varepsilon>0$. There exist points $t_1,\ldots,t_m\in K$ and a continuous partition of unity $(h_j)_{j=1}^m$ on $K$ such that the finite-rank continuous function
\[
    g(t):=\sum_{j=1}^m h_j(t)f(t_j)
\]
satisfies $\|f-g\|_\infty<\varepsilon$.

Let $(\mu_i)$ be a net in $\mathcal{P}(K)$ converging to $\mu$ in the weak-star topology. Then
\[
\int_K g\,d\mu_i-\int_K g\,d\mu=\sum_{j=1}^m \left(\int_K h_j\,d\mu_i-\int_K h_j\,d\mu\right)f(t_j) \to 0
\]
in norm. Therefore,
\begin{align*}
&\left\|\int_K f\,d\mu_i-\int_K f\,d\mu\right\| \\
&\quad \le \left\|\int_K(f-g)\,d\mu_i\right\|
+\left\|\int_Kg\,d\mu_i-\int_Kg\,d\mu\right\| +
\left\|\int_K(g-f)\,d\mu\right\| \\
&\quad \le
2\varepsilon+
\left\|\int_Kg\,d\mu_i-\int_Kg\,d\mu\right\|.
\end{align*}
This proves the desired continuity.
\end{proof}

\begin{proposition}\label{prop:moment-simplex}
Let $X$ be a Banach space and $\Phi : \mathcal{P}([0,1])\to X$ be the map from \eqref{eq:def-Phi}. Then $\Phi$ is affine, injective, and continuous. Consequently,
\[
  C:=\Phi(\mathcal{P}([0,1]))\subset X
\]
is a norm-compact convex set, and $\Phi:\mathcal{P}([0,1])\to C$ is a homeomorphism.
\end{proposition}

\begin{proof}
Affineness follows immediately from the definition. By Lemma \ref{lem:barycenter_continuity}, the continuity of $\Phi$ follows immediately from the continuity of $\phi: [0,1]\to X$. 

Next, suppose that $\Phi(\alpha)=\Phi(\beta)$. Then, for every $k\ge1$,
\[
  \int_{[0,1]} t^k\,d\alpha(t)
  =\int_{[0,1]} t^k\,d\beta(t).
\]
Since $\alpha$ and $\beta$ are probability measures, the equality also holds for $k=0$. Thus the two positive measures have the same integral against every real-valued polynomial. By the Stone--Weierstrass theorem, the real-valued polynomials are uniformly dense in $C([0,1]; \mathbb{R})$, and hence $\alpha=\beta$.

Finally, $\mathcal{P}([0,1])$ is compact in the weak-star topology, and $X$ is Hausdorff. Therefore, the continuous injection
$\Phi:\mathcal{P}([0,1])\to X$ is a homeomorphism onto its image. In particular, $C$ is norm compact.
\end{proof}

\subsection{An exposed-face renorming}
Given a Banach space $X$ over the field $\mathbb{K}$, consider $\mathbb{K}\oplus X$. Unless otherwise specified, the summand $\oplus$ is always understood as the \emph{max-sum}. Set 
\begin{equation}\label{def-F}
F:=\{(1,\Phi(\alpha)): \alpha \in \mathcal{P}([0,1])\} \subseteq \mathbb{K} \oplus X,
\end{equation}
and 
\begin{equation}\label{eq:unit-ball}
  \mathsf D
  :=\overline{\co}(\mathbb{T} F \cup(\{0\}\times B_X)) = \overline{\aco} (F \cup (\{0\} \times B_X) ), 
\end{equation}
where $\mathbb{T} := \{ \omega \in \mathbb{K} : |\omega|=1\}$. 

Note that each $e_n^*\in X^*$, $n\in\mathbb N$, admits a canonical extension to $\mathbb K\oplus X$ given by $e_n^*(t,x):=e_n^*(x)$ for every $(t,x)\in\mathbb K\oplus X.$ We also define $e_0^*:\mathbb K\oplus X\to\mathbb K$ by $e_0^*(t,x)=t$ for every $(t,x) \in \mathbb{K}\oplus X$. By a slight abuse of notation, we regard $\{e_n^*:n\in\mathbb N\cup\{0\}\}$ as a subset of $(\mathbb K\oplus X)^*$.

\begin{lemma}\label{lem:renorming}
Let $X$ be a Banach space and $\mathsf D$ be the set as in \eqref{eq:unit-ball}. Then the Minkowski functional $\mu_{\mathsf{D}}$ of $\mathsf{D}$ is an equivalent norm on $\mathbb{K}\oplus X$. Moreover, for $\omega \in \mathbb{T}$, we have $\omega F = \{ x \in \mathsf{D} : e_0^* (x) =\omega \}$. In particular, 
\[
F=\{ x \in B_{X_\mathsf D} : \operatorname{Re}e_0^* (x)=1 \}
\]
is an exposed face of $B_{X_{\mathsf D}}$.
\end{lemma}

\begin{proof}
By definition, $\mathsf D$ is closed, convex, and balanced (symmetric in the real case). The set $\mathsf{D}$ is bounded as $\Phi(\mathcal{P}([0,1]) )$ is norm compact by Proposition \ref{prop:moment-simplex}. Since $(1,0)\in F$ and $\{0\}\times B_X\subset\mathsf D$, we have
\[
  \{(t,x)\in\R\oplus X : |t|+\lVert x\rVert \le1\}
  \subset\mathsf D.
\]
The Minkowski functional $\mu_{\mathsf{D}}$ of $\mathsf D$ is therefore a norm equivalent to the original product norm, and $\mathsf D$ is its closed unit ball. For simplicity, we denote by $X_{\mathsf D}$ the space $(\mathbb{K}\oplus X, \mu_{\mathsf{D}})$.

Notice that $e_0^* \in (X_{\mathsf D})^*$ satisfies $\|e_0^* \| \leq 1.$ Since $e_0^*(1,0)=1$, it follows that $\lVert e_0^*\rVert=1$. Fix $\omega \in \mathbb{T}$. The inclusion
\[
  \omega F\subseteq\{x \in {\mathsf D}:e_0^*(x)= \omega \}
\]
is immediate. For the reverse inclusion, let $x \in \mathsf D$ with $e_0^* (x) = \omega$ be given. Write $x=(\omega, u)$ for some $u \in X$. Choose
\[
 x_m=(a_m, u_m)\in
 \operatorname{co} (
 \mathbb T F\cup(\{0\}\times B_X) ) \subseteq  \mathbb{K} \oplus X
\]
for some $a_m\in \mathbb{K}$ and $u_m \in X$, such that $x_m\to x$. We may write
\[
 x_m
 =
 \sum_{j=1}^{N_m}
 r_{m,j}\omega_{m,j}(1,c_{m,j})
 +
 q_m(0,b_m),
\]
where
\[
 r_{m,j},q_m\geq0,\quad
 \sum_{j=1}^{N_m}r_{m,j}+q_m=1, \quad \omega_{m,j}\in\mathbb T,\quad
 c_{m,j}\in \Phi(\mathcal{P}([0,1])),\quad b_m\in B_X.
\]
Put
\[
 s_m:=\sum_{j=1}^{N_m}r_{m,j}=1-q_m.
\]
Since
\[
 a_m=\sum_{j=1}^{N_m}r_{m,j}\omega_{m,j}\to\omega
\]
and $|a_m|\leq s_m\leq1$, it follows that $s_m\to1$; hence, $q_m\to0$.

Moreover,
\begin{equation}\label{eq:rmj-wmj-w-estimate}
 \sum_{j=1}^{N_m} r_{m,j}|\omega_{m,j}-\omega|^2 = 2s_m-2\operatorname{Re}(\overline{\omega}a_m) \to 0.
\end{equation}
This implies that 
\begin{align}\label{eq:rmj-omega}
&\left\| \sum_{j=1}^{N_m} r_{m,j} \omega  c_{m,j} - u_m \right\| \nonumber \\
&\quad\quad\leq  \left\| \sum_{j=1}^{N_m} r_{m,j} \omega  c_{m,j} -\sum_{j=1}^{N_m} r_{m,j} \omega_{m,j }c_{m,j} \right\| \nonumber\\
&\quad\qquad\qquad + \left\| \left( \sum_{j=1}^{N_m} r_{m,j} \omega_{m,j }c_{m,j} + q_mb_m \right)- u_m\right\| + \|q_m b_m\| \to 0, 
\end{align}
where the first summand tends to $0$ because of \eqref{eq:rmj-wmj-w-estimate} together with the fact that $\Phi(\mathcal{P}([0,1]))$ is norm-compact (in particular, norm-bounded), the second summand is indeed $0$ by definition of $u_m$, and the third summand tends to $0$ since $q_m\to 0$. For all sufficiently large $m$, the element
\[
 c_m
 :=
 \frac{1}{s_m}
 \sum_{j=1}^{N_m}r_{m,j}c_{m,j}
\]
belongs to $\Phi(\mathcal{P}([0,1]))$ because of the convexity. Moreover, 
\begin{align*}
    \|\omega c_m -u_m\| \leq \left\|\omega \left (c_m - \sum_{j=1}^{N_m} r_{m,j} c_{m,j} \right) \right\| + \left\| \sum_{j=1}^{N_m} r_{m,j}\omega c_{m,j} - u_m \right\| \to 0
\end{align*}
where the first summand tends to $0$ as $s_m \to 1$ and the second summand converges to $0$ by \eqref{eq:rmj-omega}.  Since $u_m\to u$, it follows that the element $(1,c_m) \in F$ satisfies that 
\[
\omega(1,c_m)=(\omega, \omega c_m) \to (\omega ,u)=x.
\]
Since $\omega F$ is compact, it follows that $x\in\omega F$.
\end{proof}

Throughout the paper from now on, we fix the notation $X_{\mathsf D} := (\mathbb{K}\oplus X, \mu_{\mathsf{D}})$ as in the proof of Lemma \ref{lem:renorming}.

\section{Countable support and optimal tensor representations}\label{sec:no-countable-representation}
We now use the exposed face $F$ to compare optimal integral representations with countable optimal decompositions. For a measurable map $f$ and a measure $\mu$, we write $f_\sharp\mu$ for the {pushforward} of $\mu$ under $f$.

\subsection{Couplings and optimal integral representations}
Consider $\xi : \mathcal{P} ([0,1]) \to \mathbb K \oplus X$ defined as 
\[
\xi (\alpha) := (1, \Phi(\alpha)) \in \R \oplus X \quad (\alpha \in \mathcal{P}([0,1])).
\]
For simplicity, given $t \in [0,1]$, write 
\[
p(t) := \xi (\delta_t) = (1, \Phi(\delta_t) ) = (1, \phi(t)).
\]
Observe from Lemma \ref{lem:relation-between-Phi-phi} that for every $\alpha \in \mathcal{P}([0,1])$, 
\begin{equation}\label{eq:relation-integral-xi}
    \int_{[0,1]} p(t) \, d\alpha(t)  = \left( 1, \int_{[0,1]} \phi(t) \, d\alpha(t) \right) = (1,\Phi(\alpha))= \xi (\alpha).
\end{equation}

Since the map $\phi$ is continuous, the map $t \mapsto p(t)$ is norm continuous as a map into $\mathbb{K} \oplus X$. By definition (see \eqref{def-F}), we have $p([0,1]) \subseteq F$.

 Since $\mathbb{K}\oplus X$ is linearly isomorphic to $X_{\mathsf D}$, the map $t \mapsto p(t) \in X_{\mathsf D}$ is still norm continuous; hence the following Bochner integral is well defined for every $\eta \in \mathcal{P}([0,1]^2)$:
\begin{equation*}
  \int_{[0,1]^2} p(t)\otimes p(s)\,d\eta(t,s)
  \in X_{\mathsf D} \pten X_{\mathsf D}.
\end{equation*}

\begin{proposition}\label{prop:Theta-continuous-injection}
    Let $\Theta : \mathcal{P}([0,1]^2) \to X_{\mathsf D} \pten X_{\mathsf D}$ be the map defined as
    \begin{equation}\label{eq:def-u}\tag{$\heartsuit$}
    \Theta(\eta) =   \int_{[0,1]^2} p(t)\otimes p(s)\,d\eta(t,s).
    \end{equation}
    Then $\Theta$ is an affine homeomorphism onto its image. Moreover, $\|\Theta (\eta) \|_\pi = 1$ for every $\eta \in \mathcal{P}([0,1]^2)$. 
\end{proposition}

\begin{proof}
Affineness follows immediately from the definition. Since the map 
\[
(t,s) \in[0,1]^2 \mapsto p(t)\otimes p(s) \in X_{\mathsf D} \pten X_{\mathsf D}
\]
is norm continuous, Lemma \ref{lem:barycenter_continuity} shows that $\Theta$ is continuous.

    To prove the injectivity, suppose that $\Theta(\alpha)=\Theta(\beta)$ for some $\alpha,\beta \in \mathcal{P}([0,1]^2)$. Then applying the bilinear form $(x,y)\mapsto e_j^*(x) e_k^*(y)$, we obtain, for every $j,k\ge0$, that 
    \[
    \int_{[0,1]^2} \frac{1}{2^{j+k}} t^{j}s^k \,d\alpha(t,s) = \int_{[0,1]^2} \frac{1}{2^{j+k}} t^{j} s^k \,d\beta(t,s).
    \]
By the real Stone--Weierstrass theorem, such polynomials are uniformly dense in $C([0,1]^2; \mathbb{R})$. Therefore, $\alpha=\beta$ and the injectivity of $\Theta$ is proved. 

Since $\mathcal P([0,1]^2)$ is compact in the weak-star topology and $X_D\widehat{\otimes}_\pi X_D$ is Hausdorff, $\Theta$ is a homeomorphism onto its image.

To prove the ``moreover'' part, note that $\lVert p(t)\rVert_{X_{\mathsf D}}=1$ for every $t$. Thus, 
\[
  \lVert \Theta(\eta) \rVert_\pi
  \le\int_{[0,1]^2}\lVert p(t)\otimes p(s)\rVert_\pi\,d\eta(t,s)=1.
\]
Consider the continuous bilinear form
\[
  \mathcal A: X_{\mathsf D}\times X_{\mathsf D}\to\R,
  \qquad \mathcal A(x,y)=e_0^*(x)e_0^*(y).
\]
Because $\lVert e_0^*\rVert=1$, we have $\lVert\mathcal A\rVert=1$. Since $p(t) \in F$ for every $t \in [0,1]$, it follows that $\Ree e_0^* (p(t)) = 1$ for every $t \in [0,1]$ (Lemma \ref{lem:renorming}). Therefore,  
\begin{equation}\label{eq:applying-A}
  \mathcal A(\Theta(\eta) ) =\int_{[0,1]^2}e_0^*(p(t)) e_0^* (p(s))\,d\eta(t,s) =1.
\end{equation}
It follows that $1\le\lVert \Theta(\eta) \rVert_\pi$, and hence
$\lVert \Theta(\eta) \rVert_\pi=1$.
\end{proof}

\begin{proposition}\label{prop:integral-NA}
Let $\Theta : \mathcal{P}([0,1]^2) \to X_{\mathsf D} \pten X_{\mathsf D}$ be the affine homeomorphism from Proposition \ref{prop:Theta-continuous-injection}. Then $\Theta(\eta) \in \INA_\pi(X_{\mathsf D} \pten X_{\mathsf D})$ for every $\eta \in \mathcal{P}([0,1]^2)$. More precisely, let
\[
  \gamma:[0,1]^2 \to B_{X_{\mathsf D}} \times B_{X_{\mathsf D}},
  \qquad \gamma(t,s)=(p(t),p(s)),
\]
and set $\mu:=\gamma_\sharp \eta$. Then
\begin{equation}\label{eq:u-representation}
  \Theta(\eta) =\int_{B_{X_{\mathsf D}} \times B_{X_{\mathsf D}}}x\otimes y\,d\mu(x,y), \qquad \lVert\mu\rVert=1. 
\end{equation}
Moreover, if $\eta$ is nonatomic, then $\mu$ is a nonatomic Radon probability measure.
\end{proposition}

\begin{proof}
First, we check that the representation in \eqref{eq:u-representation} is well-defined. Let $\varphi : B_{X_{\mathsf D}} \times B_{X_{\mathsf D}} \to X_{\mathsf D} \pten X_{\mathsf D}$ be the map that sends $(x,y)$ to $x \otimes y$. Note that 
\begin{equation}\label{eq:mu-has-full-meausure-on-gamma-image}
\mu (\gamma ([0,1]^2 ) ) = \eta (\gamma^{-1} ( \gamma ([0,1]^2 ) ) ) = \eta ([0,1]^2 ) = 1.
\end{equation}
Since $\gamma$ is continuous, $\gamma ([0,1]^2)$ is a compact subset of $B_{X_{\mathsf D}} \times B_{X_{\mathsf D}}$. Since $\varphi$ is continuous, it follows that $\varphi ( \gamma ([0,1]^2)) $ is a compact subset of $X_{\mathsf D} \pten X_{\mathsf D}$. This, together with \eqref{eq:mu-has-full-meausure-on-gamma-image}, shows that $\varphi$ is $\mu$-essentially separably valued. The continuity of $\varphi$ guarantees that $\varphi$ is weakly $\mu$-measurable. Therefore, Pettis's measurability theorem yields that $\varphi$ is $\mu$-measurable. Moreover, 
\[
\int_{B_{X_{\mathsf D}} \times B_{X_{\mathsf D}}} \| \varphi (x,y) \|_\pi \, d\mu(x,y) \leq \mu (B_{X_{\mathsf D}} \times B_{X_{\mathsf D}}) =1;
\]
so $\varphi$ is $\mu$-Bochner integrable and the representation in \eqref{eq:u-representation} is well-defined. 

On the other hand, since $\eta$ is a Radon probability measure on $[0,1]^2$, the pushforward measure $\mu=\gamma_\sharp \eta$ is also a Radon probability measure on $B_{X_{\mathsf D}} \times B_{X_{\mathsf D}}$ \cite[Theorem 9.1.1]{Bogachev}. By the change-of-variables formula,
\[
  \int_{B_{X_{\mathsf D}} \times B_{X_{\mathsf D}}}x\otimes y\,d\mu(x,y)
  =\int_{[0,1]^2} p(t)\otimes p(s)\,d\eta(t,s)=\Theta(\eta),
\]
and $\Theta(\eta) \in \INA_\pi (X_{\mathsf D} \pten X_{\mathsf D})$. 

Assume now that $\eta$ is nonatomic. It remains to prove that $\mu$ is nonatomic. The map $p$ is injective: if $p(t_1)=p(t_2)$, then applying the coefficient functional $e_1^*$ gives $2^{-1} t_1=2^{-1} t_2$, and hence $t_1=t_2$. Consequently, $\gamma$ is injective. Since $\gamma$ is a continuous injection from the compact space $[0,1]^2$ into the Hausdorff space $B_{X_{\mathsf D}}\times B_{X_{\mathsf D}}$, it is a homeomorphism onto its image. The measure $\mu$ is supported on $\gamma([0,1]^2)$ and satisfies $\mu=\gamma_\sharp \eta$. Therefore $\mu$ is nonatomic. \end{proof}

\subsection{Countable mixtures of product measures}

To relate product measures to elementary tensors, we shall use the fact that
the Bochner integral of a tensor-product map factors as the tensor product of
the corresponding Bochner integrals. In the present setting,
it will yield from \eqref{eq:relation-integral-xi} and \eqref{eq:def-u} that 
\begin{equation}\label{eq:Theta-is-product-of-xi}
    \Theta(\alpha\times\beta)  = \xi(\alpha)\otimes\xi(\beta) \quad (\alpha,\beta  \in \mathcal{P}([0,1])).
\end{equation}
The formula follows readily
from the Bochner--Fubini theorem and the compatibility of Bochner
integration with bounded linear operators. Although it is presumably well
known, we have not been able to locate a precise reference for it in exactly
the form needed here, so we include the short proof for completeness.

\begin{lemma}[Tensor products of Bochner integrals]\label{lem:Fubini-Bochner}
    Let $X$ and $Y$ be Banach spaces, and let $(S,\mathcal{A},\mu)$ and $(T,\mathcal{B}, \nu)$ be finite measure spaces. Suppose that $f: S \to X$ and $g: T\to Y$ are Bochner integrable. Then the map 
    \[
    F: S\times T \to X\pten Y, \quad F(s,t)=f(s)\otimes g(t),
    \]
    is Bochner integrable with respect to $\mu \times \nu$, and 
    \[
    \int_{S\times T} F(s,t) \, d(\mu\times \nu)(s,t) = \left( \int_S f(s) \, d\mu(s) \right) \otimes \left( \int_T g(t) \, d\nu(t) \right). 
    \]
\end{lemma}

\begin{proof}
    Since $f$ and $g$ are $\mu$-measurable and $\nu$-measurable, respectively, there exist sequences of simple functions $(f_n)$ and $(g_n)$ such that 
    \[
    f_n (s) \to f(s) \text{ for $\mu$-a.e.}, \text{ and } g_n(t) \to g(t) \text{ for $\nu$-a.e}.
    \]
    Note that $f_n\otimes g_n : S \times T\to X \pten Y$ is a simple function, and 
    \[
    (f_n \otimes g_n)(s,t)=f_n(s)\otimes g_n(t) \to f(s)\otimes g(t) \text{ for $\mu\times \nu$-a.e.}
    \]
    This proves that $F$ is $(\mu\times \nu)$-measurable. Moreover, by Tonelli's theorem, we have 
    \[
    \int_{S\times T} \|F(s,t)\|_\pi \, d(\mu\times\nu)(s,t) = \left(\int_S \|f(s)\|\,d\mu(s) \right) \left(\int_T \|g(t)\|\,d\nu(t) \right)<\infty.
    \]
    This shows that $F$ is Bochner integrable. 

    Next, by Bochner-Fubini's theorem (see, for instance, \cite[Proposition 1.2.7]{HvVW}), we have that 
    \begin{equation}\label{eq:Bochner-Fubini}
     \int_{S\times T} F(s,t) \, d(\mu\times \nu)(s,t) = \int_T \left( \int_S f(s)\otimes g(t)\,d\mu(s) \right) d\nu(t).
    \end{equation}
    For fixed $t \in T$, the map $R_{g(t)}: x \mapsto x \otimes g(t)$ is a bounded linear operator from $X$ into $X \pten Y$. Thus, 
    \begin{align*}
    \int_S f(s)\otimes g(t)\,d\mu(s) &= \int_S R_{g(t)} (f(s)) \, d\mu(s) \\
    &= R_{g(t)} \left( \int_{S} f(s) \, d\mu(s) \right) = \left( \int_S f(s) \,d\mu(s) \right) \otimes g(t).
    \end{align*}
    Thus, we have from \eqref{eq:Bochner-Fubini} that 
    \[
      \int_{S\times T} F(s,t) \, d(\mu\times \nu)(s,t) = \int_T \left[ \left( \int_S f(s) \,d\mu(s) \right) \otimes g(t) \right]\, d\nu(t).
    \]
    Arguing similarly with the operation $y \mapsto (\int_S f(s) \, d\mu(s)) \otimes y$, we conclude that 
    \[
     \int_{S\times T} F(s,t) \, d(\mu\times \nu)(s,t) = \left( \int_S f(s) \,d\mu(s) \right) \otimes \left( \int_T g(t) \,d\nu(t) \right). \qedhere
    \]
    \end{proof}

\begin{theorem}\label{thm:Theta-characterization}
Let $\Theta : \mathcal{P}([0,1]^2) \to X_{\mathsf D} \pten X_{\mathsf D}$ be defined as in Proposition \ref{prop:Theta-continuous-injection}.
Then, for every $\eta\in\mathcal P([0,1]^2)$,
$\Theta(\eta)\in
    \NA_\pi(X_{\mathsf D}\pten X_{\mathsf D})$ if and only if $\eta$ is a countable convex combination of product
probability measures. 
\end{theorem}

\begin{proof}
Suppose first that
\[
    \eta = \sum_{n=1}^{\infty} \lambda_n(\alpha_n\times\beta_n),
\]
where $\alpha_n,\beta_n\in\mathcal P([0,1])$, $\lambda_n\geq0$, $\sum_{n=1}^{\infty}\lambda_n=1.$ By \eqref{eq:relation-integral-xi} and Lemma~\ref{lem:Fubini-Bochner}, we have
\begin{align*}
    \Theta(\alpha_n\times\beta_n)
    &=
    \int_{[0,1]^2}
       p(t)\otimes p(s)\,
       d(\alpha_n\times\beta_n)(t,s)\\
    &=
    \left(\int_{[0,1]}p(t)\,d\alpha_n(t)\right)
    \otimes
    \left(\int_{[0,1]}p(s)\,d\beta_n(s)\right)=
    \xi(\alpha_n)\otimes\xi(\beta_n).
\end{align*}
Consequently,
\[
    \Theta(\eta)
    =
    \sum_{n=1}^{\infty}
    \lambda_n
    \xi(\alpha_n)\otimes\xi(\beta_n).
\]
Since $\xi(\alpha_n)$, $\xi(\beta_n)\in F \subseteq S_{X_{\mathsf D}}$ and $\|\Theta(\eta)\|_\pi=1$, $\Theta(\eta)\in \NA_\pi(X_{\mathsf D}\pten X_{\mathsf D}).$

Conversely, suppose that
\[
    u:=\Theta(\eta)
       \in\NA_\pi(X_{\mathsf D}\pten X_{\mathsf D}).
\]
Then there is an optimal representation
\[
    u=\sum_{n=1}^{\infty}\lambda_nx_n\otimes y_n,
    \qquad
    x_n,y_n\in B_{X_{\mathsf D}},
    \quad
    \lambda_n\geq0,
    \quad
    \sum_{n=1}^{\infty}\lambda_n=1.
\]
Applying the bilinear form $\mathcal A(x,y):=e_0^*(x)e_0^*(y)$, we obtain
\[    
1 = \operatorname{Re}\mathcal A(u) =\sum_{n=1}^{\infty}\lambda_n \operatorname{Re} e_0^*(x_n)e_0^*(y_n) \leq  \sum_{n=1}^{\infty}\lambda_n =1.
\]
Therefore, $\operatorname{Re} e_0^*(x_n)e_0^*(y_n)=1$ for every $n$ with $\lambda_n>0$.
It follows that
\[
    e_0^*(x_n)=\omega_n, \qquad  e_0^*(y_n)=\overline{\omega_n}
\]
for some $\omega_n\in\mathbb T$.
By Lemma~\ref{lem:renorming}, there exist
$\alpha_n,\beta_n\in\mathcal P([0,1])$ such that
\[
    x_n=\omega_n\xi(\alpha_n),    \qquad    y_n=\overline{\omega_n}\xi(\beta_n).
\]
Thus, by Lemma \ref{lem:Fubini-Bochner} (so, by \eqref{eq:Theta-is-product-of-xi}), $x_n\otimes y_n = \xi(\alpha_n)\otimes\xi(\beta_n) =\Theta(\alpha_n\times\beta_n)$.

Define
\[
    \nu:= \sum_{n=1}^{\infty} \lambda_n(\alpha_n\times\beta_n) \in\mathcal P([0,1]^2).
\]
Then
\[
    \Theta(\eta)=
    u = \sum_{n=1}^{\infty} \lambda_n\Theta(\alpha_n\times\beta_n) = \Theta(\nu).
\]
Since $\Theta$ is injective (Proposition \ref{prop:Theta-continuous-injection}), we obtain $\eta=\nu$. Thus, $\eta$ is a countable convex combination of product probability measures.
\end{proof}

\subsection{Diagonal couplings and proof of Theorem \ref{thm:mainA}}

Let $\Delta : [0,1]\to[0,1]^2$ be the diagonal map, i.e., $\Delta(t) = (t,t)$ for every $t \in [0,1]$. Define $\Psi : \mathcal{P}([0,1]) \to X_{\mathsf D} \pten X_{\mathsf D}$ by 
\begin{equation}\label{def:psi(alpha)}
\Psi (\alpha) := (\Theta \circ \Delta_\sharp) (\alpha) = \int_{[0,1]} p(t)\otimes p(t)\,d\alpha(t) \quad (\alpha \in \mathcal{P}([0,1]) ).
\end{equation}
Note that $\Delta_\sharp$ is affine, weak-star continuous, and injective. Since $\mathcal P([0,1])$ is compact, $\Delta_\sharp$ is a homeomorphism onto its image. Consequently, $\Psi$ is also an affine homeomorphic embedding. 

The preceding theorem converts projective norm attainment on the image of $\Theta$ into a measure-theoretic condition. To apply this characterization to diagonal couplings, we next identify the product probability measures that can be supported on the diagonal. The following elementary lemma shows that the only such measures are products of Dirac masses.
\begin{lemma}\label{lem:diagonal-product}
Let $\alpha,\beta\in\mathcal{P}([0,1])$, and suppose that
\[
  (\alpha\times\beta)(\mathfrak{D})=1, \qquad  \mathfrak{D}:=\{(t,t):t\in [0,1]\}.
\]
Then there exists $t_0\in [0,1]$ such that $\alpha=\beta=\delta_{t_0}$.
\end{lemma}

\begin{proof}
On the probability space $([0,1]^2,\alpha\times\beta)$, let
$\pi_1 (t,s)=t$ and $\pi_2 (t,s)=s$ be the coordinate maps. By assumption,
$\pi_1 = \pi_2$ on $[0,1]^2$ $(\alpha\times\beta)$-a.e. Thus, 
\begin{align*}
    \alpha(A) &= (\alpha\times \beta)(A \times [0,1]) \\
    &=(\alpha\times \beta)(\pi_1^{-1} (A)) = (\alpha\times \beta)(\pi_2^{-1} (A)) = (\alpha\times\beta) ([0,1]\times A) = \beta(A)
\end{align*}
for every Borel $A\subseteq [0,1]$. Note that 
\[
\alpha(A)^2 = \alpha(A)\beta(A) = (\alpha\times \beta)(A\times A)=(\alpha\times \beta)(\{ (t,t):t \in A \})
\]
where the last equality holds as $\alpha \times \beta$ is concentrated on the diagonal. Since $ \pi_1^{-1}(A) \setminus \{ (t,t):t \in A \} \subseteq [0,1]^2 \setminus \mathfrak{D}$ and $(\alpha\times\beta)([0,1]^2 \setminus \mathfrak{D})=0$, we have that 
\[
(\alpha\times \beta)(\{ (t,t):t \in A \}) = (\alpha\times \beta)(\pi_1^{-1} (A)).
\]
Consequently, 
\[
\alpha(A)^2 = (\alpha\times\beta)(\pi_1^{-1}(A)) = \alpha(A).
\]
So, $\alpha(A)=1$ or $0$ for all Borel $A$. Similarly, $\beta(A)=0$ or $1$ for all Borel $A$. It follows that $\alpha=\beta=\delta_{t_0}$ for some $t_0 \in [0,1]$. 
\end{proof}

\begin{proof}[Proof of Theorem \ref{thm:mainA}]

We first establish the result for the space $X_{\mathsf D}$. Let $\Psi : \mathcal{P}([0,1]) \to S_{X_{\mathsf D}\pten X_{\mathsf D}}$ be the map defined in \eqref{def:psi(alpha)}. Fix $\alpha \in \mathcal{P}([0,1])$ and set 
\[
u = \Psi (\alpha) = \Theta (\Delta_\sharp \alpha ).
\]
By Proposition \ref{prop:integral-NA}, we have $u \in \INA_\pi (X_{\mathsf D} \pten X_{\mathsf D})$.

 Suppose that $\alpha \in \mathcal{P}([0,1])$ is countably supported. That is, 
\[
\alpha = \sum_{n=1}^\infty \lambda_n \delta_{t_n}, \quad t_n \in [0,1], \quad \lambda_n\geq0 \, \text{ with }\, \sum_{n=1}^\infty \lambda_n = 1.
\]
Then $\Delta_\sharp \alpha = \sum_{n=1}^\infty \lambda_n (\delta_{t_n}\times \delta_{t_n})$ and 
\[
\Psi (\alpha) = \Theta(\Delta_\sharp \alpha)=\sum_{n=1}^\infty \lambda_n p(t_n)\otimes p(t_n). 
\]
Hence Theorem \ref{thm:Theta-characterization} shows $\Psi(\alpha) \in \NA_\pi(X_{\mathsf D}\pten X_{\mathsf D})$.

Conversely, suppose that $u = \Psi (\alpha) \in\NA_\pi(X_{\mathsf D} \pten X_{\mathsf D})$. By Theorem~\ref{thm:Theta-characterization}, there exist
$\alpha_n,\beta_n\in\mathcal P([0,1])$ and $\lambda_n\geq0$ with $\sum_n\lambda_n=1$ such that
\[
    \Delta_\sharp\alpha = \sum_{n=1}^{\infty} \lambda_n(\alpha_n\times\beta_n).
\]
Let $\mathfrak{D}:=\{(t,t):t\in[0,1]\}.$ Since $(\Delta_\sharp\alpha)(\mathfrak{D})=1$, we have
\[
    1= \sum_{n=1}^{\infty} \lambda_n(\alpha_n\times\beta_n)(\mathfrak{D}).
\]
Thus, $(\alpha_n\times\beta_n)(\mathfrak{D})=1$ for every $n$ with $\lambda_n>0$. By Lemma~\ref{lem:diagonal-product}, there exists $t_n\in[0,1]$ such that $\alpha_n=\beta_n=\delta_{t_n}$. Consequently,
\[
    \Delta_\sharp\alpha
    =
    \sum_{n=1}^{\infty}
    \lambda_n\delta_{(t_n,t_n)}, 
\]
which implies that $\alpha = \sum_{n=1}^{\infty}\lambda_n\delta_{t_n}$. Hence $\alpha$ is countably supported.

Finally, let $E$ be an arbitrary infinite-dimensional Banach space. Choose $e^*\in E^*\setminus\{0\}$ and put $X:=\ker e^*$. Then $X_{\mathsf D}$ is linearly isomorphic to $\mathbb{K}\oplus X$; hence to $E$. This completes the proof.
\end{proof}

\begin{proof}[Proof of Corollary \ref{cor:mainA}]
Let $E$ be an arbitrary infinite-dimensional Banach space. Choose $e^*\in E^*\setminus\{0\}$ and put $X:=\ker e^*$. Then $X_{\mathsf D}$ is linearly isomorphic to $E$. Take $\alpha \in \mathcal{P}([0,1])$ to be the Lebesgue probability measure on $[0,1]$. By Proposition \ref{prop:integral-NA}, 
\[
\Psi(\alpha) = \Theta (\Delta_\sharp \alpha ) 
\]
belongs to $\INA_\pi (X_{\mathsf D} \pten X_{\mathsf D})$. Since $\alpha$ is nonatomic and $\Delta$ is injective, $\Delta_\sharp \alpha$ is also nonatomic. Thus, $\Psi(\alpha)$ is witnessed by a nonatomic Radon probability measure. On the other hand, Theorem \ref{thm:mainA} shows that $\Psi (\alpha) \not\in \NA_\pi (X_{\mathsf D} \pten X_{\mathsf D})$.
\end{proof}

\section{Topological variants of integral projective norm attainment}\label{sec:topological-variants}

Integral representations may be formulated using the norm, weak, or weak-star Borel structure on the product of the unit balls. Since these Borel structures need not coincide globally, the corresponding classes could a priori differ. The key observation is that Bochner integrability itself forces the tensor-valued map to be essentially separably valued.

\subsection{Borel localization on separable supports}
For a completely regular Hausdorff space $(S, \mathcal{T})$, we write $\borel (S,\mathcal{T})$ for the $\sigma$-algebra generated by $\mathcal{T}$-open sets, and $\baire (S, \mathcal{T})$ for the $\sigma$-algebra generated by the continuous real-valued functions on $S$. That is, 
\[
\baire (S,\mathcal{T}) = \sigma \{ f^{-1} (U): f \in C(S;\mathbb{R}), \, U \subseteq \mathbb{R} \text{ open}\}.
\]

Following \cite[Section~3]{ADGJR}, let $\tau$ be a Hausdorff topology on $B_X\times B_Y$. An element $u \in X \pten Y$ is said to be \emph{$\tau$-integral projective norm-attaining} if there exists a finite positive measure $\mu$ on $\borel(B_X\times B_Y,\tau)$ such that the mapping $\varphi : (B_X\times B_Y,\tau)
 \to X \pten Y$, defined by $\varphi(x,y)=x\otimes y$, is $\mu$-Bochner integrable and satisfies 
\[
u= \int_{B_X \times B_Y} \varphi (x,y) \, d\mu(x,y) \quad \text{and} \quad \|\mu\| = \|u\|_\pi.
\]
The collection of all such tensors is denoted by $\INA_\tau (X \pten Y)$.

When $\tau$ is the product of the relative weak topologies on $B_X$ and $B_Y$, we write $\INA_w(X\widehat{\otimes}_{\pi}Y)$. Similarly, $\INA_{w^*}(X^*\widehat{\otimes}_{\pi}Y^*)$ denotes the class corresponding to the product of the relative weak-star topologies on $B_{X^*}$ and $B_{Y^*}$.

As recorded in \cite[Section~3]{ADGJR}, we always have 
\[
\INA_\pi (X \pten Y) \subseteq \INA_w (X \pten Y)
\]
and in the dual setting, 
\[
\INA_\pi (X^* \pten Y^*) \subseteq \INA_w (X^* \pten Y^*) \subseteq  \INA_{w^*} (X^*\pten Y^*).
\]
Consequently, to prove Theorem~\ref{thm:mainB}, it is enough to establish the reverse inclusions. The idea is to localize the representing measure, modulo a null set, to a product of separable subspaces. On this separable support, the relevant Borel structures coincide, thanks to the following fact.

\begin{lemma}\label{lem:Borel-Baire}
    Let $X$ be a Banach space. 
\begin{enumerate}[leftmargin=2em, label=(\roman*), ref=(\roman*)]
\itemsep0.25em
        \item If $X$ admits a Kadec norm, then $\borel (X, w)=\borel (X,\|\cdot\|)$. \label{item:Kadec} \item If $Z$ is a norm-separable subspace of $X^*$, then $\baire (Z, w^*)=\borel (Z,\|\cdot\|) .$ \label{item:Talagrand}
    \end{enumerate}
\end{lemma}
Recall that a norm on a Banach space $X$ is called a \emph{Kadec norm} if the weak and norm topologies of $X$ coincide on the unit sphere $S_X$, and recall that every separable Banach space admits an equivalent Kadec norm.

Part~\ref{item:Kadec} is due to Edgar (\cite[Corollary 2.4]{Edgar}), whereas part~\ref{item:Talagrand} is Talagrand's result (\cite[Proposition 2-2-5]{Talagrand}). Since 
\[
\baire (Z, w^*) \subseteq \borel (Z, w^*) \subseteq \borel (Z, \|\cdot\|) 
\]
in general, Lemma \ref{lem:Borel-Baire}\ref{item:Talagrand} shows that $\borel (Z, w^*) = \borel (Z,\|\cdot \|)$.

\subsection{Proof of Theorem \ref{thm:mainB}}

\begin{proof}[\text{Proof of Theorem \ref{thm:mainB}}]
Throughout the proof, every subset of a topological space is endowed with the relative topology. For $C\subseteq B_X\times B_Y$, the notation $(C,w)$ refers to the topology induced by
\[
\sigma(X,X^*)\times\sigma(Y,Y^*),
\]
whereas $(C,\|\cdot\|)$ refers to the topology induced by the product of the norm topologies. For subsets of $B_{X^*}\times B_{Y^*}$, the notation $w^*$ is understood analogously, using
\[
\sigma(X^*,X)\times\sigma(Y^*,Y).
\]

(a) We show that $\INA_w(X\pten Y)=\INA_\pi(X\pten Y).$

\noindent Let $u\in\INA_w(X\widehat{\otimes}_\pi Y) \setminus \{0\}$ and $\mu$ be a finite positive Borel measure on $(B_X \times B_Y, w)$ such that the mapping
\[
\varphi:
(B_X \times B_Y, w) \to X\widehat{\otimes}_\pi Y,
\qquad
\varphi(x,y)=x\otimes y,
\]
is $\mu$-Bochner integrable and
\[
u=\int_{B_X\times B_Y}\varphi(x,y)\,d\mu(x,y),
\qquad
\|u\|_\pi=\|\mu\|.
\]

Note that $\mu$ is concentrated on $S_X\times S_Y$. That is, 
\[
\|x\|=\|y\|=1 \quad \text{for $\mu$-almost every }(x,y)\in B_X\times B_Y.
\]
By the $\mu$-Bochner integrability of $\varphi$, there exist a
$\mu$-null Borel set
\[
N\in
\operatorname{Borel}(B_X \times B_Y, w)
\]
and a separable subspace $E$ of $X\widehat{\otimes}_\pi Y$ such that $\varphi((B_X\times B_Y)\setminus N)\subseteq E.$
Enlarging $N$ by a $\mu$-null Borel set if necessary, we may also
assume that
\[
(B_X\times B_Y)\setminus N\subseteq S_X\times S_Y.
\]

Set $A:=(B_X\times B_Y)\setminus N$. Choose a sequence $(x_n,y_n)\subseteq A$ such that $\{x_n\otimes y_n:n\in\mathbb N\}$ is norm dense in $\varphi(A)$, and put
\[
X_0:=\overline{\operatorname{span}}\{x_n:n\in\mathbb N\},
\qquad
Y_0:=\overline{\operatorname{span}}\{y_n:n\in\mathbb N\}.
\]

We claim that $A\subseteq B_{X_0}\times B_{Y_0}$. Indeed, let $(x,y)\in A$. There exists a sequence $(n_k)$ such that
\[
x_{n_k}\otimes y_{n_k}\to x\otimes y \quad \text{ in }X\widehat{\otimes}_\pi Y.
\]
If $x\notin X_0$, then by the Hahn--Banach theorem there exists $x^*\in X^*$ such that
\[
x^*\restricted_{X_0}=0 \quad\text{ and } \quad x^*(x)\neq0.
\]
Applying the bounded operator $x^*\otimes I_Y: X\widehat{\otimes}_\pi Y\to Y$ gives
\[
0= (x^*\otimes I_Y)(x_{n_k}\otimes y_{n_k}) \to (x^*\otimes I_Y)(x\otimes y) = x^*(x)y.
\]
This is impossible, since $x^*(x)\neq0$ and $\|y\|=1$. Hence
$x\in X_0$. Similarly, $y\in Y_0$, proving the claim.

By the Hahn--Banach theorem, the relative weak topologies inherited from $X$ and $Y$ coincide with $\sigma(X_0,X_0^*)$ and $\sigma(Y_0,Y_0^*)$, respectively. That is,
\[
\sigma(X, X^*) \restricted_{X_0} = \sigma(X_0, X_0^*).
\] 
Consider the separable Banach space
\[
H:=X_0\oplus_\infty Y_0.
\]
Then $B_H=B_{X_0}\times B_{Y_0}$, and the relative weak topology on $B_H$ is precisely the weak-product topology on $B_{X_0}\times B_{Y_0}$. Since $H$ is separable, it admits an equivalent Kadec norm. Hence Lemma~\ref{lem:Borel-Baire}\ref{item:Kadec} yields
\begin{equation}\label{eq:X0Y0}
\borel(B_{X_0}\times B_{Y_0},w)
=
\borel(B_{X_0}\times B_{Y_0},\|\cdot\|).
\end{equation}

Since $A \in \borel (B_X \times B_Y, w)$ and $A \subseteq B_{X_0}\times B_{Y_0}$, restriction gives 
\[
A \in \borel (B_{X_0}\times B_{Y_0}, w).
\]
Thus, \eqref{eq:X0Y0} gives that $A \in \borel (B_{X_0}\times B_{Y_0} , \|\cdot \|)$. Since $B_{X_0} \times B_{Y_0}$ is norm closed in $B_X \times B_Y$, we conclude that $A \in \borel (B_X \times B_Y, \|\cdot\|)$.

Let $\mu_A$ denote the restriction of $\mu$ to $A$. By \eqref{eq:X0Y0}, the weak-product and norm-product Borel structures on $A$ coincide, so $\mu_A$ may be regarded as a finite positive norm-Borel measure on $A$. 

Let $\iota: (A, \|\cdot\| ) \to (B_X \times B_Y, \| \cdot \|)$ be the inclusion mapping, and define
\[
\nu:=\iota_\sharp\mu_A.
\]
Then $\nu$ is a finite positive Borel measure on $(B_X \times B_Y, \| \cdot \|)$. Furthermore, $\nu$ is concentrated on $A$ and $\varphi(A)\subseteq E$. Thus $\varphi$ is $\nu$-essentially separably valued. Since $\|\varphi(x,y)\|_\pi \leq 1$ for every $(x,y)\in B_X\times B_Y$, $\nu$ is a finite measure, and $\varphi$ is norm continuous, $\varphi$ is $\nu$-Bochner integrable.

By the change-of-variables formula for pushforward measures,
\begin{align*}
\int_{B_X\times B_Y}\varphi(x,y)\,d\nu(x,y)
&=
\int_A\varphi(x,y)\,d\mu_A(x,y)\\
&=
\int_{B_X\times B_Y}\varphi(x,y)\,d\mu(x,y)=u,
\end{align*}
where the second equality follows from $\mu(N)=0$. Moreover,
\begin{align*}
\|\nu\| = \nu(B_X\times B_Y)  =
\mu_A(A) =
\mu(B_X\times B_Y) =
\|\mu\| =
\|u\|_\pi.
\end{align*}
It follows that $u\in\INA_\pi(X\widehat{\otimes}_\pi Y).$

(b) Next, we show that $\INA_{w^*} (X^*\pten Y^*)=\INA_\pi(X^* \pten Y^*).$ 

\noindent  Let $u \in \INA_{w^*} (X^* \pten Y^*) \setminus\{0\}$ and $\mu$ be a witnessing finite positive Borel measure on $(B_{X^*} \times B_{Y^*}, w^*)$ such that the mapping
\[
\varphi: (B_{X^*} \times B_{Y^*}, w^*)  \to X^* \widehat{\otimes}_\pi Y^* , \qquad \varphi(x^*,y^*)=x^*\otimes y^*,
\]
is $\mu$-Bochner integrable and
\[
u=\int_{B_{X^*}\times B_{Y^*}}\varphi(x^*,y^*)\,d\mu(x^*,y^*), \qquad \|u\|_\pi=\|\mu\|.
\]
 As above, after removing a $\mu$-null weak-star Borel set, we obtain a set $A  \subseteq S_{X^*} \times S_{Y^*}$ and a separable subspace $E$ of $X^* \pten Y^*$  such that $\varphi (  A) \subseteq E$.

Choose a sequence $(x_n^*,y_n^*) \subseteq A$ such that $\{x_n^* \otimes y_n^* : n \in \mathbb{N}\}$ is norm dense in $\varphi (A)$.  Let $W := \overline{\text{span}} \{x_n^* : n\in\mathbb{N}\}$ and $Z := \overline{\text{span}} \{y_n^* : n\in\mathbb{N}\}$. Then the same Hahn-Banach argument as above, now using elements of $X^{**}$ and $Y^{**}$, shows that 
 \[
 A \subseteq B_W \times B_Z.
 \]

Let $G:= X \oplus_1 Y$ and $H:= W\oplus_\infty Z$. Then $H$ is a norm-separable subspace of $G^*$. Moreover, $B_H = B_W \times B_Z$, and the relative weak-star topology inherited from $\sigma (G^*,G)$ is precisely the product of the relative weak-star topologies on $W$ and $Z$. Thus, by Lemma \ref{lem:Borel-Baire}\ref{item:Talagrand}, we obtain 
$$
\borel (H,w^*)=\borel (H,\|\cdot\|).
$$ 
Restricting this equality to $B_H=B_W\times B_Z$ and then to $A$, we have 
\[
\borel (A, w^*)=\borel (A,\|\cdot \|).
\]

As above, the restriction $\mu_A$ of $\mu$ to $\borel (A, w^*)$ can be regarded as a finite positive Borel measure on $(A,\|\cdot\|)$. Let $\iota : (A,\|\cdot\|) \to (B_{X^*}\times B_{Y^*}, \|\cdot \|)$ be the inclusion map and let $\nu := \iota_\sharp \mu_A$. Then $\varphi$ is $\nu$-Bochner integrable, and 
\[
u = \int_{B_{X^*}\times B_{Y^*}} \varphi (x^*,y^*) \,d\nu(x^*,y^*) \quad \text{with} \quad \|u\|_\pi = \|\nu\|.
\]
It follows that $u \in \INA_\pi (X^* \pten Y^*).$
\end{proof}

A related global separability argument appears in the recent preprint \cite{HKMRZ26}. Under the assumption that $X^*$ and $Y^*$ are separable, \cite[Lemma~4.6]{HKMRZ26} uses the coincidence of the weak-star and norm Borel structures on the whole product $B_{X^*}\times B_{Y^*}$ to obtain Bochner integrability of the canonical tensor map. Combined with an approximation-property assumption, this yields
\[
    X^*\widehat{\otimes}_\pi Y^* = \INA_\pi(X^*\widehat{\otimes}_\pi Y^*)
\]
in \cite[Corollary~4.7]{HKMRZ26}.

Theorem \ref{thm:mainB} has a different scope and conclusion. It assumes neither
separability nor the approximation property. Rather, it shows that every weak or weak-star integral representation may be replaced by a norm-Borel one. The essential separability forced by Bochner integrability allows the argument to be localized even when the corresponding Borel structures do not coincide on the ambient unit balls.

\subsection{Consequences}

Theorem \ref{thm:mainB} shows that the weak and weak-star representations introduced in~\cite{ADGJR} do not enlarge the class of integral projective norm-attaining tensors once Bochner integrability is imposed. This is a local phenomenon: the corresponding Borel structures may still differ on the ambient unit balls, as the next remark emphasizes.

\begin{remark}
    The equalities in Theorem \ref{thm:mainB} do not follow from a global coincidence of the corresponding Borel structures; in fact, these structures may differ already on a single unit ball. Indeed, Talagrand showed in \cite{Talagrand_Borel} that 
    \[
    \borel (\ell_\infty, w) \neq \borel (\ell_\infty, \|\cdot \|).
    \]
    This implies that $\borel (B_{\ell_\infty}, w) \neq \borel (B_{\ell_\infty}, \|\cdot\|)$. Moreover, since 
    \[
    \borel (B_{\ell_\infty}, w^*) \subseteq \borel (B_{\ell_\infty}, w), 
    \]
    we also have that $\borel (B_{\ell_\infty}, w^*) \neq \borel (B_{\ell_\infty}, \|\cdot\|)$.
\end{remark}

We now combine topology-independence with the existence theorem for weak-star integral representations from~\cite{ADGJR}. This yields spaces for which every tensor is integral projective norm-attaining, while Theorem~\ref{thm:mainA} still provides a tensor with no countable optimal representation.

\begin{proof}[Proof of Corollary \ref{cor:mainB}] Let $X$ be an infinite-dimensional separable reflexive Banach space with the approximation property. 
    By Corollary \ref{cor:mainA}, there exists an equivalent norm $\vertiii{ \cdot }$ on $X$ such that $Z= (X, \vertiii{\cdot })$ satisfies that $\NA_\pi (Z \pten Z) \neq  \INA_\pi (Z\pten Z)$. Equivalent renormings preserve separability, reflexivity, and the approximation property. 
    
    By \cite[Proposition 3.2]{ADGJR}, $\INA_{w^*} (Z\pten Z)=Z\pten Z$. Now, Theorem \ref{thm:mainB} yields that $\INA_{\pi} (Z\pten Z)=Z\pten Z.$
\end{proof}

\noindent \textbf{Acknowledgements}: The author was supported by the research fund of Hanyang University (HY-202500000003346).

\end{document}